\documentclass[12pt]{amsart}
\usepackage{amssymb}
\usepackage{bbm}
\usepackage{amsfonts}
\usepackage{ulem}
\usepackage{latexsym}
\usepackage[all]{xy}
\usepackage{amscd}
\usepackage{comment}
\usepackage{fullpage}
\usepackage{pstricks}
\usepackage[bookmarks=false, colorlinks, linkcolor=mylinkcolor, citecolor=mycitecolor]{hyperref}
\definecolor{mycitecolor}{rgb}{0,.6,0} 
\definecolor{mylinkcolor}{rgb}{0,0,.8}   

\numberwithin{equation}{section}
\newtheorem{theorem}[equation]{Theorem}

\newtheorem{proposition}[equation]{Proposition}

\theoremstyle{definition}
\newtheorem*{theorem-non}{Theorem}

\newtheorem{remark}[equation]{Remark}

\newcommand{\FPdim}{{\rm FPdim}}

\newcommand{\C}{\mathcal{C}}
\newcommand{\E}{\mathcal{E}}
\newcommand{\D}{\mathcal{D}}
\newcommand{\Z}{\mathcal Z}
\newcommand{\TY}{\mathcal{TY}}
\newcommand{\N}{\mathbb N}

\newcommand{\ot}{\otimes}

\newcommand{\B}{\mathcal{B}}

\newcommand{\A}{\mathcal{A}}
\newcommand{\M}{\mathcal{M}}

\newcommand{\one}{\mathbf{1}}

\newcommand{\mbbZ}{\mathbb{Z}}
\newcommand{\mcC}{\mathcal{C}}
\renewcommand{\(}{\left(}
\renewcommand{\)}{\right)}
\newcommand{\lcb}{\left\{}
\newcommand{\rcb}{\right\}}

\newcommand{\thmref}[1]{Theorem~\ref{#1}}

\DeclareMathOperator{\Rep}{Rep}

\DeclareMathOperator{\Hom}{Hom}

\DeclareMathOperator{\id}{id}
\DeclareMathOperator{\Irr}{Irr}

\DeclareMathOperator{\Stab}{Stab}
\DeclareMathOperator{\rev}{rev}

\renewcommand{\dim}{\operatorname{dim}}
\renewcommand{\Vec}{\operatorname{Vec}}

\newcommand{\Sum}{\displaystyle\sum}

\begin{document}

\title[Modular categories of dimension $pq^4$ and $p^2q^2$]
{Classification of integral modular categories of Frobenius-Perron dimension
$pq^4$ and $p^2q^2$}

\author{Paul Bruillard}
\email{paul.bruillard@math.tamu.edu}
\address{Department of Mathematics, Texas A\&M University, College Station,
Texas 77843, USA}

\author{C\'esar Galindo}
\email{cn.galindo1116@uniandes.edu.co}
\address{Departamento de matem\'aticas, Universidad de los Andes, Bogot\'a,
Colombia}

\author{Seung-Moon Hong}
\email{seungmoon.hong@utoledo.edu}
\address{Department of Mathematics and Statistics, University of Toledo, 
Ohio 43606, USA}

\author{Yevgenia Kashina}
\email{ykashina@depaul.edu}
\address{Department of Mathematical Sciences, DePaul University, Chicago,
Illinois 60614, USA}

\author{Deepak Naidu}
\email{dnaidu@math.niu.edu}
\address{Department of Mathematical Sciences, Northern Illinois
University, DeKalb, Illinois 60115, USA}

\author{Sonia Natale \and Julia Yael Plavnik}
\email{natale@famaf.unc.edu.ar}
\email{plavnik@famaf.unc.edu.ar}
\address{Facultad de Matem\'atica, Astronom\'ia y F\'isica
Universidad Nacional de C\'ordoba, CIEM--CONICET, C\'ordoba, Argentina}
\author{Eric. C. Rowell}
\email{rowell@math.tamu.edu}
\address{Department of Mathematics, Texas A\&M University, College Station,
Texas 77843, USA}

\thanks{This project began while the authors were participating in a workshop
at the \textit{American Institute of Mathematics}, whose hospitality and
support is gratefully acknowledged. 
The research of S. N. was partially supported by CONICET and
SeCyT-UNC and was carried out in part during a stay at the IH\' ES, France,
whose kind hospitality is acknowledged with thanks.
The research of J. P. was partially supported by CONICET, ANPCyT and Secyt-UNC.
The research of E. R. was partially supported by USA NSF grant DMS-1108725.}

\begin{abstract}
We classify integral modular categories of dimension $pq^4$ and $p^2q^2$ where
$p$ and $q$ are distinct primes.  We show that such categories are always
group-theoretical except for categories of dimension $4q^2$.
In these cases there are
well-known examples of non-group-theoretical categories, coming from centers of
Tambara-Yamagami categories and quantum groups.  We show that a
non-group-theoretical integral modular category of dimension $4q^2$ is 
equivalent to either one of these well-known examples or is of dimension $36$ and is twist-equivalent to fusion categories arising from a
certain quantum group.
\end{abstract}

\maketitle

\begin{section}{Introduction}
Braided group-theoretical categories are well-understood: they are equivalent to
fusion subcategories of $\Rep(D^\omega (G))$ where $G$ is a finite group and
$\omega$ is a $3$-cocycle on $G$ \cite{Na1, O}. Fusion subcategories of
$\Rep(D^\omega (G))$ are determined by triples $(K,H,B)$ where $K$, $H$ are
normal subgroups of $G$ that centralize each other,
and $B$ is a $G$-invariant $\omega$-bicharacter  on $K \times H$
\cite[Theorem 5.11]{NNW}. Triples $(K,H,B)$ for which $HK=G$ and $B$ satisfies
a certain nondegeneracy condition determine the \textit{modular}
subcategories \cite[Proposition 6.7]{NNW}.
Moreover, braided group-theoretical categories enjoy \textit{property
\textbf{F}}: the braid group representations on endomorphism spaces have
finite image \cite{ERW}.
Our approach to the classification of integral modular categories of a
given dimension is to consider those that are group-theoretical as understood
and then explicitly describe those that are not.

For distinct primes $p,q$ and $r$, any
integral modular category of
dimension $p^n$, $pq$, $pqr$, $pq^2$ or $pq^3$ is group-theoretical by
\cite{EGO,DGNO1,NR}.  On the other hand, 
non-group-theoretical integral modular categories of dimension
$4q^2$ were constructed in \cite{GNN} and \cite{NR}.  Furthermore, there are
non-group-theoretical integral modular categories of dimension $p^2q^4$ if
$p$ is odd and $p\mid\/(q+1)$, obtained as the Drinfeld centers of Jordan-Larson
categories (see \cite[Theorem 1.1]{JL}).

If $\C$ is an integral nondegenerate braided fusion category, then the set
of modular structures on $\C$ is in bijection with the set of
isomorphism classes of invertible self-dual objects of $\C$.
Thus, we view the problem of classifying integral modular categories 
as being equivalent to the problem of classifying integral nondegenerate braided fusion categories.

In this work, we classify integral
modular
categories of dimension $pq^4$ and $p^2q^2$.  In particular we prove:
\begin{theorem}
Let $\C$ be an integral modular category.
 \begin{enumerate}
  \item If $\FPdim(\C)=pq^4$ then $\C$ is
group-theoretical.
\item If $\FPdim(\C)=p^2q^2$ is \textit{odd}, then $\C$ is group
theoretical.
\item If $\mcC$ is a non-group-theoretical category of dimension $4q^2$ then
either $\mcC\cong\mathcal{E}(\zeta,\pm)$, as braided fusion categories,
with $\zeta$ an elliptic quadratic
form on $\mathbb{Z}_q\times \mathbb{Z}_q$ or $\mcC$ is twist-equivalent to
$\mcC(\mathfrak{sl}_3,q,6)$ or to $\bar{\mcC}(\mathfrak{sl}_3,q,6)$.
 \end{enumerate}
\end{theorem}

Here, $\E(\zeta,\pm)$ are modular categories
constructed in \cite{GNN}, $\mcC(\mathfrak{sl}_3,q,6)$ are the modular
categories constructed from the quantum group $U_q(\mathfrak{sl}_3)$ where $q^2$
is a primitive $6$th root of unity, 
and $\bar{\mcC}(\mathfrak{sl}_3,q,6)$ are the fusion categories 
defined in Subsection \ref{subsection 36 modular}.
The notion of twist-equivalence is defined in Subsection \ref{subsection 36 modular}. 
The $36$-dimensional categories $\bar{\mcC}(\mathfrak{sl}_3,q,6)$ are new, and will 
be investigated further in a future work.

We hasten to point out a coarser classification of these categories has been
obtained: in \cite{ENO2} it is shown that
any fusion category of dimension $p^aq^b$ is \textit{solvable}, that is, such
categories can be obtained from the category $\Vec$ of finite dimensional
vector spaces via a sequence of extensions and equivariantizations by groups
of prime order. However, the question of whether a given category admits a
$\mathbb{Z}_p$-extension or $\mathbb{Z}_p$-equivariantization is a somewhat
subtle one (see \cite{ENO3, G1}).

\end{section}
\begin{section}{Some general results}

In this section, we will recall some general results about modular categories.
These results will be used in the sections that follow.

The Frobenius-Perron dimension $\FPdim(X)$ of a simple object $X$ in a fusion category $\C$ is
defined to be the largest positive eigenvalue of the fusion matrix
$N_X$ with entries $N_{X,Y}^Z=\dim\Hom(Z, X\otimes Y)$ of $X$, that is, the
matrix representing
$X$ in the left regular representation of the Grothendieck semiring $Gr(\C)$ of $\C$.
The Frobenius-Perron dimension $\FPdim(\C)$ of the fusion category $\C$ is defined to be
the sum of the squares of the Frobenius-Perron dimensions of
(isomorphism classes of) simple objects.
A fusion category $\C$ is called {\bf integral} if $\FPdim(X)\in \N$ for all simple objects $X$.
If $\C$ is an integral modular category, then $\FPdim(X)^2$ divides $\FPdim(\C)$
for all simple objects $X \in \C$ \cite[Lemma 1.2]{EG} (see also
\cite[Proposition~3.3]{ENO1}).  Note that integral modular categories are
\textbf{pseudounitary} \cite[Proposition~8.24]{ENO1}, that is, the
Frobenius-Perron dimension coincides with the categorical dimension.

A fusion category is said to be {\bf pointed} if all its simple
objects are invertible. A fusion category whose Frobenius-Perron dimension is
a prime number is necessarily pointed \cite[Corollary~8.30]{ENO1}.

For a fusion category $\C$, the maximal pointed
subcategory of $\C$, generated by invertible objects will be denoted by
$\C_{pt}$.  A complete set of representatives of non-isomorphic simple objects
in $\C$ will be denoted $\Irr(\C)$.
The full fusion subcategory generated by all simple subobjects of
$X \ot X^*$, where $X$ runs through all simple objects of $\C$,
is called the {\bf adjoint category} of $\C$, and it is denoted by $\C_{ad}$.
If the (finite) sequence of categories $\C\supset \C_{ad}\supset
(\C_{ad})_{ad}\supset\cdots$ converges to the trivial category $\Vec$ then
$\C$ is called \textbf{nilpotent}.  Clearly, any pointed fusion category is
nilpotent.

Two objects $X$ and $Y$ in a braided fusion category $\C$ (with braiding $c$)
are said to {\bf centralize} each other if $c_{Y, X} \circ c_{X, Y} = \id_{X \otimes Y}$.
If $\D$ is a full (not necessarily fusion)
subcategory of $\C$, then the {\bf centralizer} of $\D$ in $\C$ is
the full fusion subcategory
\[
\D' := \left\{ X \in \C \mid c_{Y, X} \circ c_{X, Y}
= \id_{X \otimes Y}, \mbox{ for all } Y \in \D \right\}.
\]
If $\C' = \Vec$ (the fusion category generated by the unit object),
then $\C$ is said to be {\bf nondegenerate}.
If $\C$ is nondegenerate and $\D$ is a full fusion subcategory of $\C$, then
$(\D')'=\D$ and by  \cite{M}
$$\FPdim(\D) \cdot \FPdim(\D') = \FPdim(\C).$$

Recall that a {\bf modular category} is a nondegenerate braided fusion category equipped
with a ribbon structure.

We record the following theorems for later use.

\begin{theorem}[{\cite[Corollary 6.8]{GN}}]
\label{GN thm}
If $\C$ is a pseudounitary modular category, then $(\C_{pt})' = \C_{ad}$.
\end{theorem}

\begin{theorem}[{\cite[Corollary 4.14]{DGNO1}}]
\label{DGNO thm}
A modular category $\C$ is group-theoretical if and only if it is integral
and there is a symmetric subcategory $\mathcal{L}$ such that $(\mathcal{L}^\prime)_{ad}\subset\mathcal{L}$.
\end{theorem}

A {\bf grading} of a fusion category $\C$ by a finite group $G$ is a decomposition
\[
\C =\bigoplus_{g\in G}\, \C_g
\]
of $\C$ into a direct sum of full abelian subcategories such that the tensor product
$\otimes$ maps $\C_g\times \C_h$ to $\C_{gh}$ for all $g, \, h\in G$.
The $\C_g$'s are called {\bf components} of the $G$-grading of $\C$.
A $G$-grading is said to be {\bf faithful} if $\C_g\neq 0$ for all $g\in G$.
For a faithful grading, the dimensions of the components
are all equal \cite[Proposition 8.20]{ENO1}.
Every fusion category $\C$ is faithfully graded by
its universal grading group $U(\C)$ \cite{GN}, and this grading is called the
{\bf universal grading}.
The trivial component of this grading is $\C_e=\C_{ad}$,
where $e$ is the identity element of $U(\C)$.
For a modular category $\C$, the universal grading group $U(\C)$
is isomorphic to the group of isomorphism
classes of invertible objects of $\C$ \cite[Theorem 6.2]{GN}, in particular,
$\FPdim(\C_{pt})=|U(\C)|$.

Finally we recall some standard algebraic relations involving
the $S$-matrix $\tilde{S}$, twists $\theta_i$ and fusion constants
$N_{i,j}^k$ of a pseudounitary modular category. The  matrix $\tilde{S}$ is
symmetric and projectively unitary, with entries
given by the \textbf{twist equation}
\[
\tilde{S}_{i,j}=\theta_i^{-1}\theta_j^{-1}\Sum_kN_{i,j^*}
^k\theta_k\FPdim(X_k).
\]
The \textbf{Gauss sums} $p_{\pm}=\Sum_k \theta_k^{\pm 1}\FPdim(X_k)^2$ satisfy
$p_+p_-=\FPdim(\C)$.

\end{section}
\begin{section}{dimension $pq^4$}

\begin{theorem}
\label{pq^4}
Let $p$ and $q$ be distinct primes, and let $\C$ be an integral modular category of
dimension $pq^4$. Then $\C$ is group-theoretical.
\end{theorem}
\begin{proof}
Since $\FPdim(X)^2$ must divide $\FPdim(\C)=pq^4$ for every simple object $X \in \C$,
the possible dimensions of simple objects are $1, \, q$, and $q^2$.
Let $a, \, b$, and $c$ denote the number of isomorphism classes of simple objects
of dimension $1, \, q$, and $q^2$, respectively. We must have
$a+bq^2+cq^4=pq^4$,
and so $q^2$ must divide $a=\FPdim(\C_{pt})$. Since the dimension of a fusion
subcategory must divide $\FPdim(\C)$, it follows that there are six possible values
for $\FPdim(\C_{pt})$: $q^2, \, q^3, \, q^4, \, pq^2, \, pq^3$, and $pq^4$.

{\bf Case (i):} $\FPdim(\C_{pt}) = q^3$.
In this case, there are $q^3$ components in the grading of $\C$ by
its universal grading group $U(\C)$, and
each component has dimension $pq$. For each $g \in U(\C)$,
let $a_g,\, b_g$, and $c_g$ denote the number of isomorphism classes of simple objects
of dimension $1, \, q$, and $q^2$, respectively, contained in the component
$\C_g$.
We must have $a_g + b_g q^2 + c_g q^4 = pq$, and so $q$ must divide $a_g$.
Note that $a_g \neq 0$ for all $g \in U(\C)$, since, otherwise, we would
have $b_g q + c_g q^3 = p$, a contradiction. Thus, each component must contain
at least $q$ (non-isomorphic) invertible objects, and since there are $q^3$
components, it follows that there must be at least $q^4$ (non-isomorphic)
invertible objects, a contradiction.

{\bf Case (ii-v):} $\FPdim(\C_{pt}) \in  \{ q^4, pq^2, pq^3,  pq^4\}$.
In each case, $\FPdim(\C_{ad})$ is a power of a prime, so $\C_{ad}$ is nilpotent
\cite{GN}. Consequently, $\C$ is also nilpotent, and since it is integral and
modular it follows that it is group-theoretical \cite{DGNO1}.

{\bf Case (vi):} $\FPdim(\C_{pt}) = q^2$.
In this case, $\FPdim (\C_{ad}) = pq^2$. This fact together with
the possibilities for the dimensions of simple objects implies that
$(\C_{ad})_{pt}$ must be of dimension $q^2$, and so $\C_{pt} \subseteq \C_{ad}$.
Hence $\C_{pt}$ is symmetric, since $\C_{ad} = (\C_{pt})'$.

\medbreak We claim that $\C_{pt}$ is a Tannakian subcategory.
This is true if $q$ is odd \cite[Corollary 2.50]{DGNO2}.
So, assume that $q = 2$, and suppose on the contrary that $\C_{pt}$ is not Tannakian. Then $\C_{pt}$ contains
a symmetric subcategory $\mathcal S$ equivalent to the category of super vector spaces.
Let $g \in \mathcal S$ be the unique nontrivial (fermionic) invertible object, and let $\mathcal S' \subseteq \C$ denote the M\" uger centralizer of $\mathcal S$.
By \cite[Lemma 5.4]{mueger}, we have $g\otimes X \ncong X$, for any simple object $X \in \mathcal S'$.

On the other hand, we have $\C_{pt} \subseteq \mathcal S'$ and  $\FPdim
(\mathcal S') = 8p$. The possibilities for the  dimensions of simple objects of
$\C$
 imply that the number of simple objects of dimension $2$ in $\mathcal S'$ is
necessarily odd.
Therefore, the action by tensor multiplication of the group of invertible objects of $\C$ on the set of isomorphism classes of simple objects of FP-dimension $2$
of $\mathcal S'$ must have a fixed point, which is a contradiction. Hence $\C_{pt}$ is Tannakian, as claimed.

\medbreak Therefore  $\C_{pt} \cong \Rep(G)$ as symmetric tensor categories,
where $G$ is a group of order $q^2$.
Let $\widehat \C : = \C_G$ denote  the corresponding de-equivariantization of
$\C$. By the main result of \cite{K, mueger-galois}, $\widehat  \C$ is a
$G$-crossed braided
fusion category (of dimension $pq^2$), and the equivariantization of $\widehat
\C$ with respect to the associated $G$-action is equivalent to $\C$ as braided
tensor categories (see \cite[Theorem 2.12]{GNN}).

Furthermore, since $\C$ is modular, the associated $G$-grading of $\widehat
 \C$ is faithful \cite[Remark 2.13]{GNN}.
Thus, the trivial component $\widehat{\C}_e \supseteq \widehat{\C}_{ad}$  of this grading
is of dimension $p$ and, in particular, it
is pointed.
Hence $\widehat \C$ is a nilpotent fusion category.

In view of \cite[Corollary
5.3]{GN}, the square of the Frobenius-Perron dimension of a simple object of
$\widehat \C$ must divide $\FPdim(\widehat{\C}_e)=p$. Since $\widehat \C$
is also integral, we see that $\widehat \C$ is itself pointed.
It follows from \cite[Theorem 7.2]{NNW} that $\C$, being an equivariantization
of a pointed fusion category, is group-theoretical.
\end{proof}

\begin{remark}
In this remark, we show that two of the four cases addressed in
Case (ii-iv) of the proof above can not occur.
If $\FPdim(\C_{pt}) = q^4$, then $\FPdim((\C_{pt})') = p$,
so $(\C_{pt})'$ must be pointed. Therefore, $(\C_{pt})'$ is contained in
$\C_{pt}$, and this implies that $p$ divides $q^4$, a contradiction.
If $\FPdim(\C_{pt}) = pq^3$, then each component in the universal grading of $\C$
has FP-dimension $q$, and so it can not accommodate a non-invertible object,
a contradiction.
\end{remark}

\end{section}
\begin{section}{dimension $p^2q^2$}
In this section, we will make repeated use of the following result.
\begin{theorem}[{\cite[Theorem 1.6 and Proposition 4.5(iv)]{ENO2}}]
\label{ENO2 nontrivial invertible}
If $p$ and $q$ are primes and $\mcC\neq\Vec$ is a fusion category of dimension
$p^{a}q^{b}$ then $\mcC$ contains a non-trivial invertible object.
\end{theorem}
\begin{theorem}
\label{p2q2}
Let $p<q$ be primes, and let $\mcC$ be an integral modular category of dimension $p^{2}q^{2}$. Then one of the following is true:
\begin{enumerate}
\item $\mcC$ is group-theoretical.
\item $p=2$, $q=3$, and $\FPdim\(\mcC_{pt}\)=3$

\item $p\mid q-1$ and $\FPdim\(\mcC_{pt}\)=p$.
\end{enumerate}
\end{theorem}
\begin{proof}

Since $\FPdim\(X\)^{2}$ must divide $\FPdim\(\mcC\)=p^2q^2$
for every simple object $X\in \mcC$,
the possible dimensions of the simple objects are $1$, $q$, and
$p$. Since $\mcC_{pt}$ is a fusion subcategory of $\mcC$ we know that
$\FPdim\(\mcC_{pt}\) \mid p^2q^2$, and hence
$\FPdim\(\mcC_{pt}\)\in\lcb1,p,q,p^{2},q^{2}, pq, pq^{2}, p^{2}q,
p^{2}q^{2}\rcb$. Applying \thmref{ENO2 nontrivial invertible}, we conclude that
$\FPdim\(\mcC_{pt}\)>1$. The proof now proceeds by cases based on
$\FPdim\(\mcC_{pt}\)$. For each $g \in U(\C)$, let $a_{g}$,
$b_{g}$, and $c_{g}$ denote the number of isomorphism classes of
simple objects of dimension $1$, $p$, and $q$ in the component $\mcC_{g}$,
respectively. Let $e$ denote the identity element of $U(\C)$.

\textbf{Case (i-v):} $\FPdim\(\mcC_{pt}\) \in \{p^{2}q^{2},pq^{2}, p^2q, p^2,
q^2\}$.
In each case, $\FPdim(\C_{ad})$ is a power of a prime, so $\C_{ad}$ is nilpotent
\cite{GN}. Consequently, $\C$ is also nilpotent, and since it is modular it
follows that
it is group-theoretical \cite[Corollary 6.2]{DGNO1}.


\textbf{Case (vi):} $\FPdim\(\mcC_{pt}\)=pq$. In this case, $\FPdim\(\mcC_{g}\)=pq$ for all $g\in U\(\mcC\)$. Since $p<q$ we immediately conclude that $c_{g}=0$ for all $g \in U(\C)$ and thus $a_{g}\neq 0$ from the equation $pq=a_{g}+b_{g}p^{2}$. By \thmref{ENO2 nontrivial invertible}, we know that $\mcC_{ad}$ contains a non-trivial invertible object. Since there are $pq$ components, the number of invertible objects is at least $pq+1$, a contradiction.

\textbf{Case (vii):} $\FPdim\(\mcC_{pt}\)=q$. In this case, $\FPdim\(\mcC_{g}\)=p^{2}q$ for all $g\in U\(\mcC\)$. We can apply \thmref{ENO2 nontrivial invertible} to $\mcC_{ad}$ to deduce that $\mcC_{pt}\subset\mcC_{ad}$.
Examining the dimension of $\mcC_{ad}$ we have $p^{2}q=q+b_{e}p^{2}+c_{e}q^{2}$.
So, $q$ must divide $b_{e}$, and hence $b_{e}=0$. Consequently, $c_{e}q=\(p-1\)\(p+1\)$.
Since $p<q$ we must have $p=2$ and $q=3$.

\textbf{Case (viii):} $\FPdim\(\mcC_{pt}\)=p$. As in case (vii) we immediately
conclude that $\mcC_{pt}\subset\mcC_{ad}$, and $\FPdim\(\mcC_{g}\)=pq^{2}$ for
all $g\in U\(\mcC\)$. Examining the dimension of $\mcC_{ad}$ we have
$pq^{2}=p+b_{e}p^{2}+c_{e}q^{2}$. Therefore, $c_{e}=0$ and
$b_{e}=\frac{q^{2}-1}{p}$. Similar analysis in the nontrivial components reveals
that $b_{g}=0$ and $c_{g}=p$. We will identify the simple objects of
$\mcC_{pt}$ with the elements of $U(\C)$ and denote them by $g^{k}$, ordered such
that $g^{k}\otimes g^{\ell}=g^{k+\ell}$ (exponents are computed modulo $p$). We
will denote the objects of dimension $p$ by $Y_{r}$ and the objects of dimension
$q$ in the $\mcC_{g^{k}}$ component by $X_{i}^{k}$.

We will show that $p\mid q-1$. This is
immediate in the case that $p=2$, so we will assume that $p\geq 3$. We first
need to determine some of the fusion rules. Denote by
$\Stab_{U(\C)}\(Y_{r}\)$ the stabilizer of the object $Y_r$ under the tensor
product action of $g^j\in\C_{pt}$. Computing the dimension of
\begin{align*}
Y_{r}\otimes Y_{r}^{*}=\bigoplus_{h\in \Stab_{U(\C)}\(Y_{r}\)}h\oplus
\bigoplus_{s=1}^{\frac{q^{2}-1}{2}}N_{Y_{r},Y_{r}^{*}}^{Y_{s}}Y_{s}
\end{align*}
 we see that $p$ must divide $|\Stab_{U(\C)}\(Y_{r}\)|$ and hence
$\Stab_{U(\C)}\(Y_{r}\)=U(\C)$. An analogous argument shows that
$\Stab_{U(\C)}\(X_{i}^{k}\)$ is trivial for all $i$ and $k$. In particular, the
action of $U(\C)$ on $\mcC_{g^{k}}$ is fixed-point free and so we may
relabel such that $g\otimes X_{i}^{k}=X_{i+1}^{k}$ (with indices computed modulo
$p$). These results about the stabilizers allow us to compute
$N_{Y_{r},X_{i}^{k}}^{X_{j}^{k}}$ as follows
\begin{align*}
\bigoplus_{j=1}^{p}N_{Y_{r},X_{i}^{k}}^{X_{j}^{k}}X_{j}^{k}&=Y_{r}\otimes X_{i}^{k}=\(g^{\ell}\otimes Y_{r}\)\otimes X_{i}^{k}=Y_{r}\otimes X_{i+\ell}^{k}=\bigoplus_{j=1}^{p}N_{Y_{r},X_{i+\ell}^{k}}^{X_{j}^{k}}X_{j}^{k}.
\end{align*}
Since this must hold for all $\ell$ we can conclude that
$N_{Y_{r},X_{i}^{k}}^{X_{j}^{k}}=N_{Y_{r},X_{h}^{k}}^{X_{j}^{k}}$ for all
$r,h,i,j$, and $k$. A dimension count gives
$N_{Y_{r},X_{i}^{k}}^{X_{j}^{k}}=1$.

Denote by $\tilde{S}$ the $S$-matrix of $\C$, and the
entries by $\tilde{S}_{A,B}$ (normalized so that $\tilde{s}_{\one,\one}=1$).
Since $\FPdim(Y_r)=p$ and $\FPdim(X_i^k)=q$ are coprime, \cite[Lemma 7.1]{ENO2}
implies that either $\tilde{s}_{Y_{r},X_{i}^{k}}=0$ or
$|\tilde{s}_{Y_{r},X_{i}^{k}}|=pq$.  Since the columns of $\tilde{S}$ have
squared-length $(pq)^2$ we must have $\tilde{s}_{Y_{r},X_{i}^{k}}=0$.

 We compute $\tilde{s}_{Y_{r},X_{i}^{k}}$ another way using the fusion rules
above and the twist equation to conclude that
$0=\Sum_{j=1}^{p}\theta_{X_{j}^{k}}$.
The vanishing of this sum allows us to compute the Gauss sums as follows.
\begin{align*}
p_{+}&=\Sum_{Z\in
\Irr\(\mcC\)}\theta_{Z}\FPdim(Z)^{2}=p\(1+p\Sum_{r=1}^{\frac{q^{2}-1}{2}}
\theta_{ Y_ { r}}\)
\end{align*}
However, \cite[Proposition 5.4]{Ng} shows that $\mcC$ is anomaly free and in
particular that $p_{+}=pq$. From this it immediately follows that
\begin{align*}
\frac{q-1}{p}&=\Sum_{r=1}^{\frac{q^{2}-1}{p}}\theta_{Y_{r}}.
\end{align*}
The right hand side of this equation is an algebraic integer and so we
conclude that $p$ must divide $q-1$.
\end{proof}

\begin{remark}
In this remark, we show that two of the five cases addressed in
Case (i-v) of the proof above can not occur.
If $\FPdim\(\mcC_{pt}\)=p^{2}$, then $\FPdim\(\mcC_{ad}\)=q^{2}$,
so applying \thmref{ENO2 nontrivial invertible} we see that there is a
non-trivial invertible object in $\mcC_{ad}$. Therefore, $a_{e}\geq2$ and
$a_{e}\mid q^{2}$. On the other hand, the invertible objects in $\mcC_{ad}$
will form a fusion subcategory of $\mcC_{pt}$, and so $a_{e}\mid p^{2}$,
a contradiction. A similar argument shows that the case
$\FPdim\(\mcC_{pt}\)=q^{2}$ can not occur.
\end{remark}

Next, we recall a general fact about modular categories.
Let $\mcC$ be a modular category and suppose that it
contains a Tannakian subcategory $\E$. Let $G$ be
a finite group such that $\E
\cong \Rep(G)$, as symmetric categories. The de-equivariantization $\C_G$ is
a braided $G$-crossed fusion category of Frobenius-Perron dimension $\FPdim
(\C)/|G|$ (see \cite{K, mueger-galois}).

Since $\C$ is modular, the associated $G$-grading of $\C_G$ is faithful and
the trivial component $\C_G^0$ is a modular category of Frobenius-Perron
dimension $\FPdim
(\C)/|G|^2$. Furthermore, as a consequence of \cite[Corollary 3.30]{DMNO} we have
an equivalence of braided fusion categories
\begin{equation}\label{equiv} \C \boxtimes (\C_G^0)^{\rev} \cong \Z(\C_G).
\end{equation}
Notice that this implies that $\C$ is group-theoretical if and only if $\C_G$ is
group-theoretical \cite[Proposition 3.1]{witt-wgt}.

\begin{proposition}\label{g-crossed} Let $p<q$ be prime numbers. Let $\mcC$
be an integral modular category of dimension $p^{2}q^{2}$ and let $G\cong U(\C)$
be the group
of invertible objects of $\mcC$.  Suppose that $\mcC$ is
not group-theoretical. Then there exists  a
$G$-crossed braided fusion category  $\widehat  \C$
such that the equivariantization of $\widehat
\C$ with respect to the associated $G$-action is equivalent to $\C$ as braided
fusion categories. The corresponding $G$-grading on $\widehat  \C$ is
faithful, the trivial component $\widehat  \C_e$ is a modular
category of Frobenius-Perron dimension $\FPdim
(\C)/|G|^2$, and there is an equivalence of braided fusion categories
\begin{equation} \mcC \boxtimes (\widehat  \C_e)^{\rev} \cong \Z(\widehat  \C).
\end{equation}
Moreover, $\widehat \C$ is not
group-theoretical.
\end{proposition}

\begin{proof} In view of the preceding comments, it will be enough to show that
the category $\E = \C_{pt}$ is a Tannakian
fusion subcategory.

By Theorem \ref{p2q2}, we may assume that  $|G| = p$ or $\FPdim (\mcC) =
36$ and $|G| =
3$. Let $\D \subseteq \C$ be a nontrivial fusion subcategory.
Since the order of $G$ is a prime number, it follows from Theorem~\ref{ENO2
nontrivial invertible} that $\C_{pt} \subseteq \D$, so $\D_{pt} = \C_{pt} =
\E$.
In particular, $\C_{pt} \subseteq \C_{ad} = \C_{pt}'$ and therefore $\C_{pt}$ is
symmetric. If the order of $G$ is odd, this implies that $\E$ is Tannakian.

So, we may assume that $|G| = 2$.
Suppose on the contrary that $\E$ is not
Tannakian. Then $\E$ is equivalent, as a symmetric category, to the category
$\operatorname{sVec}$ of finite-dimensional super vector spaces.
 Therefore $\E'$ is a
slightly degenerate fusion category of Frobenius-Perron dimension $2q^2$. Let $g \in \E$ be the unique nontrivial invertible object.
By \cite[Lemma 5.4]{mueger}, we have $g\otimes X \ncong X$, for any simple
object $X \in \mathcal \E'$.

The possible Frobenius-Perron dimensions of simple
objects of $\C$ in this case are $1$, $2$, and $q$. This leads to the equation
$\FPdim (\E') = 2q^2 = 2 + 4 a + q^2b$, where $a, b$ are non-negative integers.
If $b \neq 0$, then $\E'$ contains a Tannakian subcategory $\mathcal B$, by
\cite[Proposition 7.4]{ENO2}.
Hence, in this case, $\E$ is Tannakian, since $\E \subseteq \mathcal B$.

Otherwise, if $b=0$, every non-invertible simple object $X$ of $\E'$ is of Frobenius-Perron
dimension $2$ and
therefore the stabilizer $\Stab_G(X)$ of any such object under the action of
the
group $G$ by tensor multiplication is not trivial, as follows from the relation
$$X \otimes X^* \cong \bigoplus_{g \in \Stab_G(X)}g \oplus
\bigoplus_{Y\in\Irr(\C)\setminus \Irr(\C_{pt})} N_{X,X^*}^Y
Y.$$
Then we see that the action
of $G$
on the set of isomorphism classes of simple objects of Frobenius-Perron dimension
$2$
of $\E'$ must be trivial, which is a contradiction. Hence $\E$
is Tannakian, as claimed.
\end{proof}

\begin{remark} Keep the notation in Proposition \ref{g-crossed}. Suppose that
$\C$ is not group-theoretical and $\FPdim (\C_{pt}) = p < q$. Then
$\FPdim(\widehat \C) = pq^2$ and $\widehat \C_e$ is a modular category of
Frobenius-Perron dimension $q^2$, and hence it is pointed. Since  $\widehat \C_{ad}
\subseteq \widehat \C_e$, $\widehat \C$ is a nilpotent integral fusion
category.
Then $\FPdim (X) = 1$ or $q$, for all simple object $X$ of $\widehat \C$
\cite[Corollary 5.3]{GN}.
Moreover, since $\widehat \C$ is not pointed, it is of type $(1, q^2; q,
p-1)$ (that is, having $q^2$ non-isomorphic simple objects of dimension $1$ and
$p-1$ non-isomorphic simple objects of dimension $q$.)
\end{remark}

\begin{theorem}\label{impar} Let $2< p<q$ be prime numbers, and let $\mcC$ be an
integral
modular category of dimension $p^{2}q^{2}$. Then  $\C$ is group-theoretical.
\end{theorem}

\begin{proof}
By Theorem \ref{p2q2}, we may assume that $\FPdim (\C_{pt}) = p$ and $p \mid
q-1$.
Keep the notation in Proposition \ref{g-crossed}. The category $\widehat \C$ has
Frobenius-Perron dimension $pq^2$.
Observe that $\widehat \C$ must be group-theoretical. Otherwise, by
\cite[Theorem 1.1]{JL} we should have $p \mid q+1$, leading to the contradiction
$p = 2$. Therefore Proposition \ref{g-crossed} implies that $\C$ is
group-theoretical, as claimed.
\end{proof}

Let $A = \mathbb Z_q \times \mathbb Z_q$, with $2\neq q$ prime, let $\zeta$ be an
elliptic
quadratic form on $A$, and let $\tau = \pm \frac{1}{q}$. Then the associated
Tambara-Yamagami fusion categories
$\TY(A, \zeta, \tau)$ are inequivalent and not group-theoretical. By
\cite[Theorem 1.1]{JL}, these are the only non-group-theoretical fusion
categories of dimension $2q^2$.

Examples of  non-group-theoretical modular categories $\C$ of Frobenius-Perron
dimension $4q^2$ such that $\FPdim (\C_{pt}) = 2$ were constructed in
\cite[Subsection 5.3]{GNN}; these
examples consist of two equivalence classes, denoted
$\E(\zeta, \pm)$, according to the sign
choice of $\tau = \pm \frac{1}{q}$. By construction, there is an embedding of
braided fusion categories $\E(\zeta, \pm) \subseteq \Z(\TY(A, \zeta, \tau))$.

\begin{theorem}\label{2q2} Let $q\neq 2$ be a prime number, and let $\mcC$ be an
integral modular
category such that  $\FPdim (\mcC) = 4q^2$ and $\FPdim\(\mcC_{pt}\)= 2$.
Then either $\C$ is group-theoretical or $\mcC \cong \E(\zeta, \pm)$ as braided fusion categories.
\end{theorem}

\begin{proof} Keep the notation in Proposition \ref{g-crossed} and suppose that
$\C$ is not group-theoretical. Then
$\FPdim (\widehat \C) = 2q^2$ and $\widehat \C$ is not group-theoretical.
Hence, by  \cite[Theorem 1.1]{JL}, $\widehat \C \cong \TY(A, \zeta, \tau)$ as
fusion categories, where $\zeta$ is an elliptic quadratic form on $A = \mathbb
Z_q \times \mathbb Z_q$, and $\tau = \pm \frac{1}{q}$.
In view of Proposition \ref{g-crossed}, we have an equivalence of braided fusion
categories
\begin{equation}\label{prod}
 \mcC \boxtimes (\widehat \C_e)^{\rev} \cong \Z(\TY(A, \zeta, \tau)),
\end{equation}
where $\widehat \C_e$ is a pointed modular category of Frobenius-Perron
dimension $q^2$.

The center of $\TY(A, \zeta, \tau)$ is described in \cite[Section 4]{GNN}. The
group of invertible objects of
$\Z(\TY(A, \zeta, \tau))$ is of order $2q^2$. In particular, $\Z(\TY(A, \zeta,
\tau))$ contains a unique pointed fusion subcategory $\B$ of dimension $q^2$,
which is nondegenerate.
We note that, since the M\" uger centralizer $\E(\zeta, \pm)'$ inside of
$\Z(\TY(A, \zeta, \tau))$ is of dimension $q^2$, whence pointed, this implies
that $\E(\zeta, \pm) = \B'$.

We must have $\B = (\widehat \C_e)^{\rev}$. Hence, by \eqref{prod}, $\mcC
\cong
\B' = \E(\zeta, \pm)$, finishing the proof.
\end{proof}

Using Theorem \ref{impar} and Theorem \ref{2q2},
we can now strengthen Theorem \ref{p2q2}:

\begin{theorem}\label{lastpqtheorem}
Let $p<q$ be primes, and let $\mcC$ be an integral modular category of dimension $p^{2}q^{2}$. Then one of the following is true:
\begin{enumerate}
\item $\mcC$ is group-theoretical.
\item $p=2$, $q=3$, and $\FPdim\(\C_{pt}\)=3$.
\item $p=2$, $\FPdim\(\C_{pt}\)=2$, and $\C \cong \E(\zeta, \pm)$,
as braided fusion categories, for some
elliptic quadratic form $\zeta$ on $\mbbZ_q \times \mbbZ_q$.
\end{enumerate}
\end{theorem}

\begin{remark} In view of \cite{JL} there are three equivalence classes of
non-group-theoretical integral fusion categories of Frobenius-Perron
dimension 36. The argument in the proof of Theorem 4.9 implies that a
non-group-theoretical integral modular category that satisfies
Theorem \ref{lastpqtheorem}(b) is equivalent to a fusion
subcategory of the center of one of these.
\end{remark}

In the subsection below, we investigate further the non-group-theoretical categories that
satisfy Theorem \ref{lastpqtheorem}(b).

\begin{subsection}{Modular categories of dimension 36}
\label{subsection 36 modular}

We begin by classifying the possible fusion rules corresponding to
non-group-theoretical modular categories satisfying the conditions of Theorem
\ref{lastpqtheorem}(b).

\begin{proposition}\label{36lemma} Let $\mcC$ be a non-group-theoretical integral
modular category of dimension $36$ with $\Irr(\C_{pt})=\{\one,g,g^2\}$.  Then:
\begin{enumerate}
 \item $\mcC=\mcC_{0}\oplus \mcC_{1}\oplus \mcC_{2}$ as a
$\mathbb{Z}_3$-graded fusion category with respective isomorphism classes of
simple objects
$$\{\one,g,g^2,Y\}\cup\{X,gX,g^2X\}\cup\{X^*,gX^*,g^2X^*\},$$
 where $\FPdim(g^i)=1$, $\FPdim(g^iX)=2$ and
$\FPdim(Y)=3$.

\item Up to relabeling $g\leftrightarrow g^{-1}$ the fusion rules are
determined by:
 \begin{eqnarray}
 g\otimes Y&\cong& Y, \quad Y^{\otimes 2}\cong \one\oplus g\oplus
g^2\oplus 2Y\notag\\
\label{fusion36} g^i\otimes X&\cong& g^iX,\quad Y\otimes X\cong X\oplus gX\oplus
g^2X\\
X\otimes X^*&\cong&
\one\oplus Y\notag
\end{eqnarray}
and either:
\begin{enumerate}
\item[(i)] $X^{\ot 2}\cong X^*\oplus gX^*$, or
\item[(ii)] $X^{\ot 2}\cong g^2X^*\oplus gX^*$.
\end{enumerate}
\end{enumerate}
\end{proposition}
\begin{proof}
 First note that $\mcC$ is faithfully $\mbbZ_3$-graded, so that each graded
component has dimension $12$ and simple objects can only have dimension $1,2$
or $3$.  Solving the Diophantine equations provided by the dimension formulas
(observing that $\mcC$ is not pointed) we see that
$\mcC_{pt}\subset\mcC_0=\mcC_{ad}$ which gives us the dimensions and objects
described in (a).

The fusion rules given in (\ref{fusion36}) are determined by
the
dimensions and the symmetry rules for the fusion matrices. The remaining fusion
rules will be determined from $X^{\otimes 2}$.
Clearly $X^{\otimes 2}\in\mcC_{2}$, so
$$X^{\otimes 2}\cong a_0X^*\oplus a_1gX^*\oplus a_2g^2X^*$$
where $\sum_i a_i=2$.  We claim no $a_i=2$.  For suppose
$X^{\otimes 2}\cong 2g^iX^*$ for some $0\leq i\leq 2$.  Then $(X^*)^{\otimes
2}\cong 2g^{-i}X$ and so
$$(X\otimes X^*)^{\otimes 2}\cong 4(g^i\otimes g^{-i})\otimes X\otimes X^*\cong
4 (\one\oplus Y).$$ On the other hand $X\otimes X^*\cong \one\oplus Y$ so
$$(X\otimes X^*)^{\otimes 2}\cong (\one\oplus Y)^{\otimes 2}\cong
2\one\oplus g\oplus g^2\oplus 4Y,$$
a contradiction.
Therefore $X^{\otimes 2}$ is multiplicity free.  This leaves 3 possibilities:
1) $a_0=1$ and $a_1=1$ or 2) $a_0=1$ and $a_2=1$ or 3) $a_0=0$ and $a_1=a_2=1$.
The first two are
equivalent under the labeling change $g\leftrightarrow g^2$ proving (b).
\end{proof}

\begin{remark}
The non-group-theoretical integral modular categories
$\mcC(\mathfrak{sl}_3,q,6)$ have fusion rules as in Proposition \ref{36lemma} (b)(i). The category $\mcC(\mathfrak{sl}_3,q,6)$ is obtained from the quantum group $U_q(\mathfrak{sl}_3)$ with $q^2$ a primitive $6$th root of unity.  The data of this category and a proof of non-group-theoreticity may be found in \cite[Example 4.14]{NR}.
\end{remark}


Next, we classify, up to equivalence of fusion categories, modular categories
realizing the fusion rules described in Proposition \ref{36lemma}.
To this end, we will need the notion of twist-equivalence, defined next.
Let $G$ be a finite group, and let $e$ denote its identity element.
Given a $G$-graded fusion category $\mcC:=(\mcC,\otimes, \alpha)$ and a 3-cocycle $\eta\in Z^3(G,\mathbb{C}^*)$, the natural isomorphism $$\alpha_{X_\sigma,X_\tau,X_\rho}^\eta=\eta(\sigma,\tau,\rho)\alpha_{X_\sigma,X_\tau,X_\rho}, \  \  \  (X_\sigma\in \mcC_\sigma, X_\tau\in \mcC_\tau,  X_\rho\in \mcC_\rho, \sigma,\tau,\rho \in G),$$
defines a new fusion category $\mcC^\eta:=(\mcC,\otimes, \alpha^\eta)$. The fusion categories $\mcC$ and $\mcC^\eta$ are  equivalent as $G$-graded fusion categories if and only if the cohomology class of $\eta$ is zero, see \cite[Theorem 8.9]{ENO3}.  We shall say that two $G$-graded fusion categories $\mcC$ and $\mathcal{D}$ are {\bf twist-equivalent} if there is a $\eta\in Z^3(G,\mathbb{C}^*)$ such that $\mcC^\alpha$ is $G$-graded equivalent to $\mathcal{D}$ (compare with \cite{KW}).


If  $(\C, c)$ is a $G$-graded strict braided fusion category, then each  $g\in
(\mcC_e)_{pt}$ defines a $\mcC_e$-bimodule equivalence
$L_g:\mcC_\sigma\to \mcC_\sigma, X\mapsto g\otimes X$ with natural isomorphism
$c_{g,V}\otimes \id_X:L_g(V\otimes X)\to V\otimes L_g(X)$, for every $\sigma\in
G$.

Let $\mcC=\mcC(\mathfrak{sl}_3,q,6)$  and consider the normalized symmetric
2-cocycle $\chi:\mathbb{Z}_3\times \mathbb{Z}_3 \to \pi$  defined by
$\chi(1,1)=\chi(1,2)=g^2, \chi(2,2)=1$, where
$\pi=\Irr(\mcC_{pt})= U(\mcC)=\{\one, g,g^2\}$.
We define a new tensor product $\bar{\otimes}:\mcC\boxtimes \mcC\to \mcC$  as $$\bar{\otimes}\vert_{\mcC_{i}\boxtimes \mcC_{j}}=L_{\chi(i,j)}\circ \otimes.$$
Since $\text{H}^4(\mathbb{Z}_3,\mathbb{C}^*)=0$, it follows by \cite[Theorem 8.8]{ENO3} that we can find isomorphisms $$\omega_{i,j,k}:\chi(i+j,k) \otimes \chi(i,j)\to \chi(i,j+k)\otimes \chi(j,k)$$ such that  the natural isomorphisms
$$
\hat{\alpha}_{X_i,X_j,X_k}^\omega=(\id_{\chi(i,j+k)}\otimes c_{\chi(j,k),X_i}\otimes \id_{X_k})\circ(\omega_{i,j,k}\otimes \id_{X_i\otimes X_j\otimes X_k}),
$$define an associator with respect to $\bar{\otimes}$ and we get  a new $\mathbb Z_3$-graded fusion category 
\[
\bar{\mcC}(\mathfrak{sl}_3,q,6):=(\mcC,\bar{\otimes},\hat{\alpha}^\omega).
\]
\begin{remark}
The notation $\bar{\mcC}(\mathfrak{sl}_3,q,6)$ is ambiguous, because we are not specifying $\omega$. However, $\bar{\mcC}(\mathfrak{sl}_3,q,6)$ is unique up to  twist-equivalence.
\end{remark}

\begin{theorem} \label{twist-equivalence}
Let $\A$ be a  fusion category.
\begin{itemize}
\item[(a)] If $\A$ has fusion rules given by Proposition \ref{36lemma} (b)(i),
then $\A$ is twist-equivalent  to $\mcC(\mathfrak{sl}_3,q,6)$ for some choice
of $q$.

\item[(b)] If $\A$ is braided and has fusion rules given by Proposition
\ref{36lemma} (b)(ii), then $\A$ is  twist-equivalent to
$\bar{\mcC}(\mathfrak{sl}_3,q,6)$ for some choice of $q$.
\item[(c)] Any fusion category twist-equivalent to  $\mcC(\mathfrak{sl}_3,q,6)$ or $\bar{\mcC}(\mathfrak{sl}_3,q,6)$ is non-group-theoretical.
\end{itemize}
\end{theorem}
\begin{proof}
(a) This follows from the classification results in \cite[Theorem $A_\ell$]{KW}.\\

\noindent (b) Since $\bar{\mcC}(\mathfrak{sl}_3,q,6)= \mcC(\mathfrak{sl}_3,q,6)$ as
abelian categories, their simple objects are the same. However, since the tensor
product is different the duals of  simple objects can be different, so we shall
use the following notation for the simple objects of
$\bar{\mcC}(\mathfrak{sl}_3,q,6)$:
\[
\one,g,g^2, Y, gX, g^2X,\bar{X}, g\bar{X},g^2\bar{X},
\]
where $\bar{X}=X^*$, with respect to the original tensor
product of $\mcC(\mathfrak{sl}_3,q,6)$.

Next,  we investigate the  fusion rules of $\bar{\mcC}(\mathfrak{sl}_3,q,6)$. First note that $X^*\in \mcC_{2}$ and $X\bar{\otimes}\bar{X}=g^2\oplus Y$, so $X^*=g\bar{X}$ and $X\bar{\otimes}X^*= \one\oplus Y$. Since $\chi$ is normalized, the only important fusion rule
that changes is
$$
X\bar{\otimes} X= g^2\otimes( \bar{X}\oplus g\bar{X})=g^2\bar{X}\oplus \bar{X}=gX^*\oplus g^2X^*.
$$
Note that the fusion rules of $\bar{\mcC}(\mathfrak{sl}_3,q,6)$ are the same as Proposition \ref{36lemma} (b)(ii).

If $\A$ is a braided fusion category with fusion rules given by Proposition
\ref{36lemma} (b)(ii), then using the 2-cocycle $\chi^{-1}$, we can construct a
fusion category $\mathcal{D}$ with the same fusion rules of 
$\mcC(\mathfrak{sl}_3,q,6)$. By (a), there exists a $q$ such that $\D$ is
twist-equivalent to $\mcC(\mathfrak{sl}_3,q,6)$ and again using the 2-cocycle
$\chi$ on $\mcC(\mathfrak{sl}_3,q,6)$ we get a fusion category  twist-equivalent
to $\A$.\\

\noindent (c) Let $G$ be a finite group, and let $e$ denote its identity element.
In \cite[Theorem 1.2]{G2}, it was proved that a $G$-graded fusion category
$\A$ is group-theoretical if and only if there is a pointed $\A_e$-module
category $\M$ such that $\A_{\sigma}\boxtimes_{\A_e}\M\cong \M$ as $\A_e$-module
categories for all $\sigma \in G$. If $\A$ is twist-equivalent to
$\mcC(\mathfrak{sl}_3,q,6)$ or $\bar{\mcC}(\mathfrak{sl}_3,q,6)$, then we have $\A_e= \mcC(\mathfrak{sl}_3,q,6)_e$ as fusion categories and $\A_\sigma=
\mcC(\mathfrak{sl}_3,q,6)_\sigma$ as $\A_e$-bimodule category. Since
$\mcC(\mathfrak{sl}_3,q,6)$ is non-group-theoretical \cite[Example 4.14]{NR},  $\A$ is
also non-group-theoretical.
\end{proof}


\end{subsection}
\begin{subsection}{Conclusions}
In this section, we have classified integral modular categories of dimension 
$p^2q^2$ up to equivalence of braided fusion categories 
with the exception of non-group-theoretical $\mathbb{Z}_3$-graded
36-dimensional modular categories, which
are classified only up to equivalence of fusion categories.

It can be shown that the category $\bar{\mcC}(\mathfrak{sl}_3,q,6)$
is not of the form $\Rep(H)$ for a Hopf
algebra $H$ using the same technique as in \cite[Theorem 5.27]{GHR}.  
More generally, a
non-group-theoretical fusion category $\C$ with fusion rules as in 
Proposition \ref{36lemma}
cannot be equivalent to a category of the form $\Rep H$, $H$ a Hopf algebra:
If $\C \cong \Rep(H)$ for some Hopf algebra $H$,
then since $\C$ admits a faithful $\mathbb Z_3$-grading,
we would have a central exact sequence $k^{\mathbb Z_3} \to H \to \overline H$,
where $\dim \overline H = 12$ and $\Rep \overline H \cong \C_0$. The
classification of semisimple Hopf algebras of dimension 12 implies that
$\overline H$ is a group algebra. Hence $k^{\mathbb Z_3} \to H \to \overline
H$ is an abelian exact sequence and therefore $H$ is group-theoretical,
a contradiction.

The
existence of a modular structure on the category 
$\bar{\mcC}(\mathfrak{sl}_3,q,6)$
will be discussed in a future
work, 
but for the interested reader we provide the
modular data where $q=e^{\pi i/3}$:

$$S:=\begin{pmatrix} 1&1&1&3&2&2&2&2&2&2
\\ \noalign{\medskip}1&1&1&3&2\,{q}^{2}&2\,{q}^{-2}&2\,{q}^{2}&2\,{q}^
{-2}&2\,{q}^{2}&2\,{q}^{-2}\\ \noalign{\medskip}1&1&1&3&2\,{q}^{-2}&2
\,{q}^{2}&2\,{q}^{-2}&2\,{q}^{2}&2\,{q}^{-2}&2\,{q}^{2}
\\ \noalign{\medskip}3&3&3&-3&0&0&0&0&0&0\\ \noalign{\medskip}2&2\,{q}
^{2}&2\,{q}^{-2}&0&2\,{q}^{-1}&2\,q&2\,q&2\,{q}^{-1}&-2&-2
\\ \noalign{\medskip}2&2\,{q}^{-2}&2\,{q}^{2}&0&2\,q&2\,{q}^{-1}&2\,{q
}^{-1}&2\,q&-2&-2\\ \noalign{\medskip}2&2\,{q}^{2}&2\,{q}^{-2}&0&2\,q&
2\,{q}^{-1}&-2&-2&2\,{q}^{-1}&2\,q\\ \noalign{\medskip}2&2\,{q}^{-2}&2
\,{q}^{2}&0&2\,{q}^{-1}&2\,q&-2&-2&2\,q&2\,{q}^{-1}
\\ \noalign{\medskip}2&2\,{q}^{2}&2\,{q}^{-2}&0&-2&-2&2\,{q}^{-1}&2\,q
&2\,q&2\,{q}^{-1}\\ \noalign{\medskip}2&2\,{q}^{-2}&2\,{q}^{2}&0&-2&-2
&2\,q&2\,{q}^{-1}&2\,{q}^{-1}&2\,q\end{pmatrix}$$

\bigskip

$$T:=\mathrm{Diag}(1,1,1,-1,{q}^{2},{q}^{2},1,1,{q}^{-2},{q}^{-2})
$$
\end{subsection}

\end{section}


\end{document}